\documentclass{amsart}
\usepackage{graphicx}
\usepackage{hyperref}
\newtheorem{thm}{Theorem}
\newcommand{\R}{\mathbb R}
\begin{document}

\title[Parameterising Besicovitch Sets]{A note on the parameterisation of Besicovitch sets}%
\author{Toby C.\ O'Neil}%
\address{Department of Mathematics and Statistics\\ The Open University\\ Walton Hall\\ Milton Keynes MK7 6AA\\ UK}%
\email{t.c.oneil@open.ac.uk}%
\date{\today}
\thanks{}%
\subjclass{28A75}%
\keywords{Besicovitch set, Kakeya set, line segment, Baire-1}%

\begin{abstract}
We show that there is a set of Lebesgue measure zero in the plane within which a line segment can be rotated by a Baire-1 map.
\end{abstract}
\maketitle

In his blog post `A note on the Kakeya needle problem', Terry Tao~\cite{tao} gave a simple argument to show that there is no set of Lebesgue measure zero in the plane within which a line segment can be rotated continuously. In this short note, we show that there is a set of Lebesgue measure zero within which a line segment can be rotated by a map that is Baire-1 (a pointwise limit of a sequence of continuous functions).

To achieve this, it is enough for us to find a Baire-1 map that parameterises all line segments with slope $b\in [0,1]$ and maps them into a fixed (measurable) set of zero area, since a Baire-1 map that parameterises all possible directions can then be created by `stitching' 8 such maps together.

The key is to exploit the duality described by Falconer in Chapter~12 of~\cite{falconer}. This leads us to define a map $l\colon  \R^2\to\mathcal{K}(\R^2)$ (the non-empty compact subsets of the plane) by setting $l(a,b)$ to be the line segment of length 1 that starts at $(a,0)$ and has slope $b$.

We start by observing that $l$ is a continuous map from $[0,1]\times [0,1]$ to $\mathcal{K}(\R^2)$ when $\mathcal{K}(\R^2)$ is given the usual Hausdorff metric $d_H$. Indeed, for $a,b,a',b'\in [0,1]$,
$$d_H (l(a,b),l(a',b'))\leq C(|a-a'|+|b-b'|),$$
for some constant $C$ independent of $a,a',b,b'$.

\begin{thm}
  There is a Baire-1 map $f\colon [0,1]\to \R^2$ such that $\bigcup \{f(s):s\in [0,1]\}$ is a subset of the plane with zero area  and for each $s$, $f (s)$ is a unit line segment of slope $s$.
\end{thm}

\begin{proof}
  Let $F$ be the one-dimensional self-similar subset of the unit square determined by dividing the unit square into 16 equal subsquares and retaining the second square of the first row, the last square of the second row, the first square of the third row and the third square of the fourth row. Thus each row and column contains exactly one subsquare and the resulting self-similar set has full projection onto both the $x$- and $y$-axes, has positive and finite 1-dimensional Hausdorff measure, and is purely 1-unrectifiable. Consequently $L(F)=\{a+bx: (a,b)\in F,\, x\in\R\}$ contains lines crossing $[0,1]$ of every slope i $[0,1]$. Also, by Proposition~12.1 of~\cite{falconer},  $L(F)$ has zero area.

  Define $g\colon [0,1]\to\R$ by
  $$g(t)=\inf\{x: (x,t)\in F\}.$$
  Then it is easy to see that $g$, being the left envelope of the compact set $F$, is lower semi-continuous  and hence, in particular, it is Baire-1.
 Finally let $f\colon [0,1]\to \R^2$ be given by
  $$f(t)= (l\circ g)(t),$$
  then $f$ is also Baire-1, since it is the composition of a continuous map with a Baire-1 map.

  It remains only to observe that $\bigcup \{f(s):s\in [0,1]\}$ contains a line segment of every slope in $[0,1]$ and has zero Lebesgue measure since $\bigcup \{f(s):s\in [0,1]\}\subseteq L(F)$.
\end{proof}

\bibliographystyle{amsplain}

\end{document}